\pdfoutput=1    

\documentclass[11pt,a4paper]{article} 

\usepackage[OT2,T1]{fontenc}
\usepackage{amsmath}
\usepackage{amssymb,verbatim}
\usepackage{amsfonts}
\usepackage{amsthm}
\usepackage{pifont}
\usepackage{graphicx}
\usepackage{datetime}
\usepackage{scrdate}
\usepackage{enumitem}
\usepackage{geometry}
\usepackage{float}
\usepackage{bigints}
\usepackage{multirow}

\usepackage{hyperref}

\hypersetup{colorlinks=true, linkcolor=red, citecolor=red, pdfstartview=FitH, linktocpage=false}

\newtheorem{thm}{Theorem}[section]

\theoremstyle{definition}

\numberwithin{equation}{section}

\newcommand{\beq}{\begin{equation*}}
\newcommand{\eeq}{\end{equation*}}
\newcommand{\beqn}{\begin{equation}}
\newcommand{\eeqn}{\end{equation}}
\newcommand{\dd}{\mathrm{d}}
\newcommand{\ii}{\mathrm{i}}
\newcommand{\ee}{\mathrm{e}}

\newcommand{\ph}{\varphi}
\newcommand{\E}{\mathbb{E}}
\newcommand{\W}{\mathbb{P}}
\newcommand{\Var}{\mathbb{V}\mathrm{ar}}
\newcommand{\EARk}{\mathbb{E}\big[A_R^k\big]}
\newcommand{\EAk}{\mathbb{E}\big[A_\ka^k\big]}
\newcommand{\EURk}{\mathbb{E}\big[U_R^k\big]}
\newcommand{\EUk}{\mathbb{E}\big[U_\ka^k\big]}
\newcommand{\ka}{\varkappa}
\newcommand{\rh}{\varrho}
\newcommand{\Sigm}{\varSigma}
\newcommand{\Om}{\varOmega}

\begin{document}

\title{Probabilistic properties of the elliptic motion}
\author{Uwe B\"asel}
\date{} 
\maketitle
\thispagestyle{empty}
\begin{abstract}
\noindent In this paper we consider the plane elliptic motion which occurs if the moving centrode is a circle of radius $r$ and the fixed centrode a circle of radius $2r$. Every point of the moving plane generates an ellipse in the fixed plane. Let a disk of radius $R$, $0 \le R < \infty$, concentric to the moving centrode be attached to the moving plane. If a point $P$ is chosen at random from this disk, then the area and the perimeter of the ellipse generated by $P$ are random variables. We determine the moments and the distributions of these random variables for the case that $P$ is uniformly distributed over the area of the disk.\\[0.2cm]
\\[0.2cm]
\end{abstract}

\section{Introduction}

We consider a fixed Euclidean plane $S$ with a Cartesian frame 
of origin $O$ and $x,y$ axes, and a moving Euclidean plane $\Sigm$ with a Cartesian frame 
of origin $\Om$ and $\xi,\eta$ axes. To $\Sigm$ we attach a circle $C_1$ with centre $\Om$ and radius $r$; to $S$ we attach a circle $C_0$ with centre $O$ and radius $2r$ (see Fig.\ \ref{Fig:Ell_motion}).

If $C_1$ rolls inside $C_0$, then every point $P$ fixed in $\Sigm$ generates an ellipse in $S$. Therefore, this motion is called {\em elliptic motion}. If $P = \Om$, then the ellipse is a circle with centre $O$ and radius $r$; if $P\in C_1$, then the ellipse degenerates to a (double) line segment of length $4r$ which is one diameter of $C_0$. \cite[pp.\ 2-3, 8-9]{Krause_Carl}, \cite[pp.\ 14-15]{Blaschke_Mueller}  

We denote by $\ph$ the angle between the $x$-axis and the line segment $
\overline{\rule{0pt}{2.75mm}{} O\Om}$. W.l.o.g.\ we assume that for $\ph = 0$ the $\xi$-axis lies on the $x$-axis and both axes have equal direction. Then the equation in complex form of the ellipse generated by $P$ is given by 
\beqn \label{Eq:Complex_form}
  X 
= X(\ph) = r\,\ee^{\ii\ph} + \varXi\,\ee^{-\ii\ph}\,,\quad
  0 \le \ph < 2\pi\,,
\eeqn
with
\begin{itemize}[leftmargin=0.75cm] \setlength{\itemsep}{-3pt}
\item[1)] $\varXi := \rh\,\ee^{\ii\alpha}$, where $\rh$, $\alpha$ are the polar coordinates of $P$ with respect to the $\xi,\eta$-frame, or
\item[2)] $\varXi := \xi+\ii\eta$, where $\xi$, $\eta$ are the Cartesian coordinates of $P$ with respect to the $\xi,\eta$-frame.
\end{itemize}       


\noindent
In the first case, from \eqref{Eq:Complex_form} we get
\beqn \label{Eq:Param_Represent1}
\left.
\begin{array}{l@{\;=\;}c@{\;}c@{\;}l}
  x & r\cos\ph & + & \rh\cos(\ph-\alpha)\,,\\[0.1cm]
  y & r\sin\ph & - & \rh\sin(\ph-\alpha)
\end{array}
\right\}
\eeqn
as parametric representation of the ellipse, and in the second case,
\beqn \label{Eq:Param_Represent2}
\left.
\begin{array}{l@{\;=\;}c@{\;}c@{\;}c@{\;}c@{\;}l}
  x & r\cos\ph & + & \xi\cos\ph & + & \eta\sin\ph\,,\\[0.1cm]
  y & r\sin\ph & + & \eta\cos\ph & - & \xi\sin\ph\,.
\end{array}
\right\}
\eeqn

\begin{figure}[H]
  \begin{center}
	\includegraphics[scale=0.7]{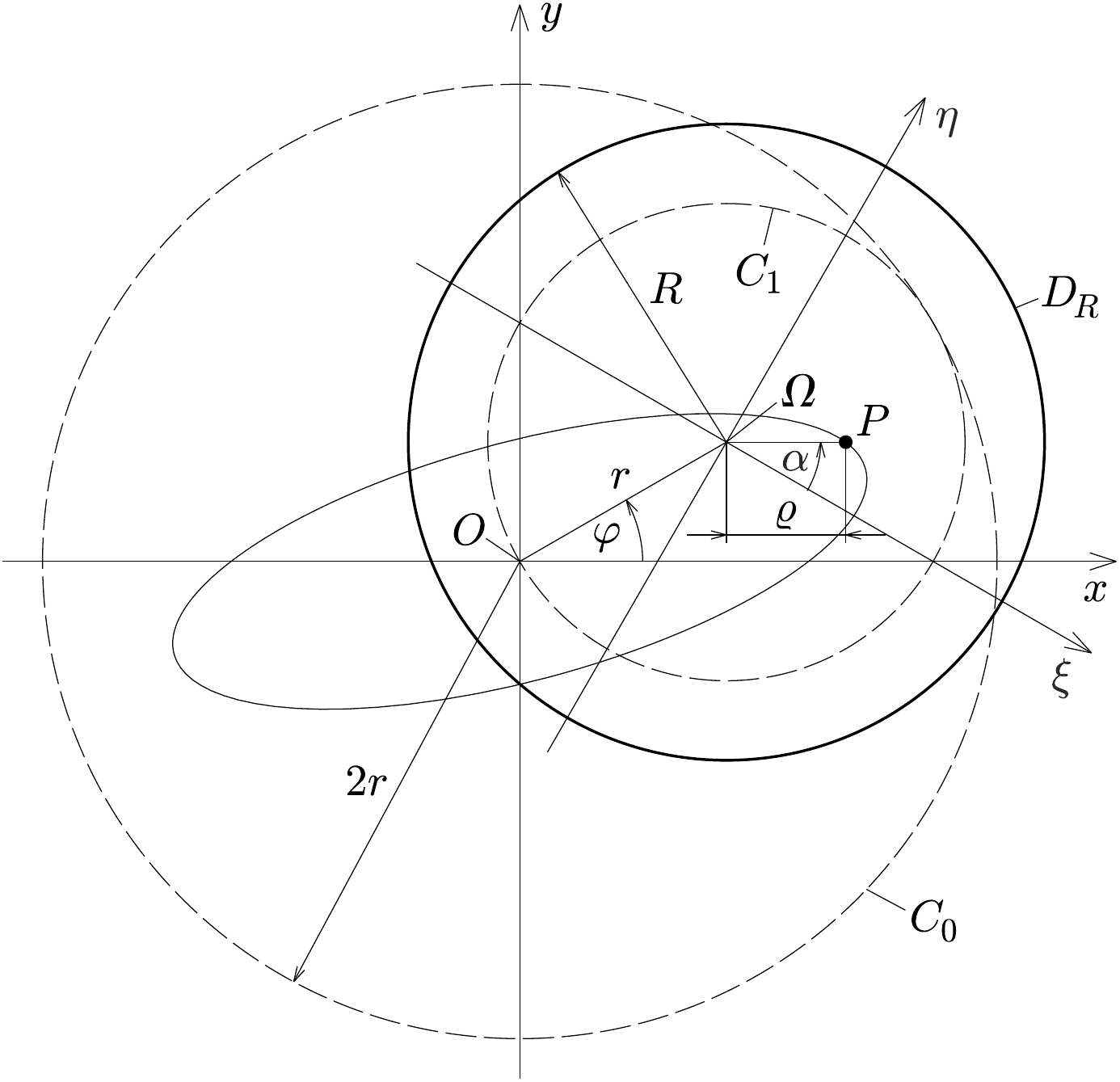}
  \end{center}
  \vspace{-0.4cm}
  \caption{\label{Fig:Ell_motion} Elliptic motion}
  \vspace{0.6cm}
  \begin{center}
	\includegraphics[scale=0.31]{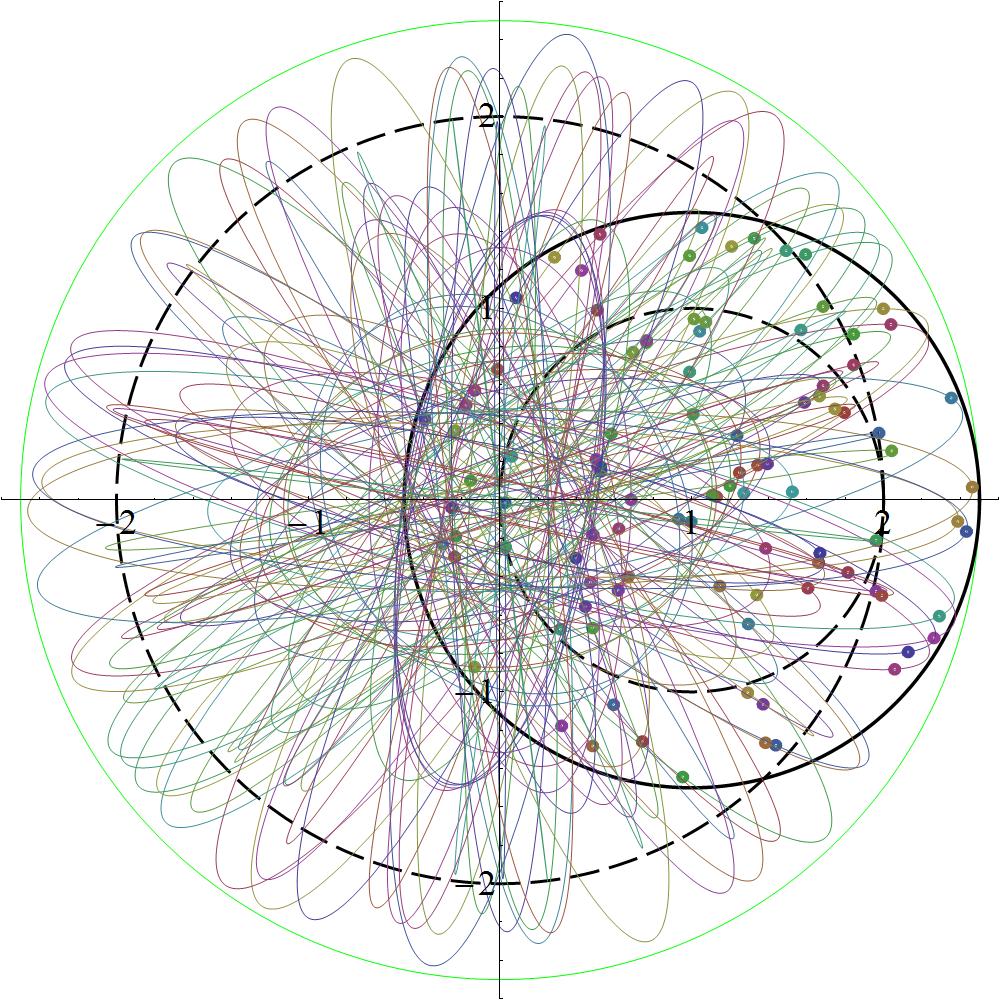}
  \end{center}
  \vspace{-0.4cm}
  \caption{\label{Fig:Random_Ellipses01} 100 random ellipses and their generating points; $r=1$, $R=1.5$}
\end{figure}

\newpage
\noindent
From \eqref{Eq:Param_Represent1} one finds that the length of the semi-major axis is equal to $r+\rh$, and the length of the semi-minor axis equal to $|r-\rh|$. Hence all points $P \in \Sigm$ with equal distance $\rh$ from $\Om$ generate congruent ellipses. The centres of all ellipses lie in $O$. The angle between the $x$-axis and the major-axis of an ellipse is equal to $\alpha/2$. 

$C_0$ is the fixed centrode and $C_1$ the moving centrode of the elliptic motion. It is possible to get the equations of $C_0$ and $C_1$ backwards from the equation(s) of the motion \eqref{Eq:Complex_form}, \eqref{Eq:Param_Represent1} or \eqref{Eq:Param_Represent2}. \cite[p.\ 8-9]{Krause_Carl}, \cite[p.\ 14-15]{Blaschke_Mueller} (For the notions of the fixed and the moving centrode see e.\,g \cite[pp.\ 257-259]{Bottema_Roth}.) 

Now we consider a disk 
\beqn \label{Eq:Disk}
  D_R
= \left\{(\xi,\eta)\in\Sigm\colon 0\le\xi^2+\eta^2\le R^2\right\},\quad
  0 \le R < \infty\,.
\eeqn
If $P$ is chosen at random from $D_R$, then the area and the perimeter of the ellipse generated by $P$ are random variables which we denote by $A_R$ and $U_R$, respectively. In this paper we determine the moments and the distributions of these random variables for the case that $P$ is uniformly distributed over the area of $D_R$. Since $r$ is no essential parameter, we identify $A_\ka \equiv A_R$ and $U_\ka \equiv U_R$, where $\ka = R/r$ 

Fig.\ \ref{Fig:Random_Ellipses01} shows the result of a simulation with 100 random points and ellipses.

\section{Moments of the area}

The moments of the random area $A_R$ enclosed by the ellipse generated by a random point $\in D_R$ (see \eqref{Eq:Disk}) are given by
\beq
  \EARk
= \left(\int_{P\in D_R}A^k\,\dd P\right)\bigg/
  \left(\int_{P\in D_R}\dd P\right),
\eeq
where $A$ is the area of the ellipse generated by the point $P \in D_R$, and $\dd P = \dd\xi \wedge \dd\eta$ is the density for sets of points in the plane. Up to a constant factor this density is the only one that is invariant under motions \cite[p.\ 13]{Santalo}.  Due to the point symmetry we use the polar coordinates $\rh,\alpha$. With $\xi = \rh\cos\alpha$, $\eta = \rh\sin\alpha$ we get
\begin{align*}
  \dd\xi
= {} & \frac{\partial\xi}{\partial\rh}\,\dd\rh
+ \frac{\partial\xi}{\partial\alpha}\,\dd\alpha
= \cos\alpha\,\dd\rh - \rh\sin\alpha\,\dd\alpha\,,\\[0.1cm]
  \dd\eta
= {} & \frac{\partial\eta}{\partial\rh}\,\dd\rh
+ \frac{\partial\eta}{\partial\alpha}\,\dd\alpha
= \sin\alpha\,\dd\rh + \rh\cos\alpha\,\dd\alpha
\end{align*}
and 
\begin{align*}
  \dd\xi \wedge \dd\eta
= {} & \cos\alpha\sin\alpha\;\dd\rh\wedge\dd\rh
+ \rh\cos^2\alpha\;\dd\rh\wedge\dd\alpha
- \rh\sin^2\alpha\;\dd\alpha\wedge\dd\rh
- \rh^2\sin\alpha\cos\alpha\;\dd\alpha\wedge\dd\alpha\\
= {} & \rh\,\dd\rh\wedge\dd\alpha\,,
\end{align*}
hence
\beqn \label{Eq:ExpectedArea01}
  \EARk
= \frac{1}{\pi R^2}\int_{\alpha=0}^{2\pi}\,\int_{\rh=0}^R
  A^k(\rh)\,\rh\,\dd\rh\,\dd\alpha
= \frac{2}{R^2}\int_{\rh=0}^R A^k(\rh)\,\rh\,\dd\rh\,.
\eeqn

The area enclosed by an ellipse generated by a point $P\in D_R$ with distance $\rh$ from $\Om$ is given by
\beqn \label{Eq:A(rh)}
  A(\rh)
= \left\{\begin{array}{l@{\quad\mbox{if}\quad}l}
	\pi\left(r^2-\rh^2\right) & 0\le\rh\le r\,,\\[0.1cm]
	\pi\left(\rh^2-r^2\right) & r<\rh<\infty\,.
  \end{array}\right.
\eeqn 
With $w=\rh/r$, the area \eqref{Eq:A(rh)} function may be written as
\beqn \label{Eq:A^*(w)}
  \widetilde{A}(w)
= \left\{\begin{array}{l@{\quad\mbox{if}\quad}l}
	\pi r^2\left(1-w^2\right) & 0\le w\le 1\,,\\[0.1cm]
	\pi r^2\left(w^2-1\right) & 1<w<\infty\,.
  \end{array}\right.
\eeqn

\begin{thm}
The $k$-th moment, $k = 1,2,\ldots$, of the random area $A_\ka \equiv A_R$, $\ka = R/r$, of an ellipse generated by a random point $P \in D_R$ ($P$ uniformly distributed over the area of the disk $D_R$) is given by
\beq
  \EAk
= \left\{
  \begin{array}{l@{\quad\mbox{if}\quad}l}
	\pi^k r^{2k} & \ka = 0\,,\\[0.2cm]
	\dfrac{\pi^k r^{2k}}{k+1}\,\dfrac{1-(1-\ka^2)^{k+1}}{\ka^2} &
		0<\ka\le 1\,,\\[0.3cm]
	\dfrac{\pi^k r^{2k}}{k+1}\,\dfrac{1+(\ka^2-1)^{k+1}}{\ka^2} & 
		1 < \ka < \infty\,.	
  \end{array}
  \right.
\eeq
\end{thm}

\begin{proof}
First, we consider the case $0 < \ka \le 1$ ($0 < R \le r$). Eq. \eqref{Eq:ExpectedArea01} becomes
\beq
  \EARk
= \frac{2\pi^k}{R^2}\int_0^R\left(r^2-\rh^2\right)^k\rh\:\dd\rh\,.
\eeq
The substitution $\rh = rw$ gives
\beq
  \EARk
= \frac{2\pi^k r^{2k+2}}{R^2}\int_{w=0}^{R/r}\left(1-w^2\right)^k 
	w\:\dd w\,,
\eeq
which, with $A_\ka \equiv A_R$ may also be written as
\beq
  \EAk
= \frac{2\pi^k r^{2k}}{\ka^2}\int_{w=0}^{\ka}\left(1-w^2\right)^k 
	w\:\dd w\,.
\eeq
The substitution
\beq
  y = 1 - w^2\,,\quad \dd y = -2w\,\dd w\,,\quad
  \dd w = -\frac{\dd y}{2w}
\eeq
yields
\begin{align*}
  \EAk
= {} & -\frac{\pi^k r^{2k}}{\ka^2}\int_{y=1}^{1-\ka^2} y^k\:\dd y
= \frac{\pi^k r^{2k}}{\ka^2}\int_{1-\ka^2}^1 y^k\:\dd y
= \frac{\pi^k r^{2k}}{\ka^2(k+1)}\:y^{k+1}\,\bigg|_{1-\ka^2}^1\\
= {} & \frac{\pi^k r^{2k}}{k+1}\:\frac{1-\left(1-\ka^2\right)^{k+1}}
	{\ka^2}\,.
\end{align*}
Applying L'H\^opital's rule we get
\beq
  \lim_{\ka\rightarrow 0}\frac{1-\left(1-\ka^2\right)^{k+1}}{\ka^2}
= \lim_{\ka\rightarrow 0}\frac{(k+1)\left(1-\ka^2\right)^k 2\ka}{2\ka}
= (k+1)\lim_{\ka\rightarrow 0}\left(1-\ka^2\right)^k
= k+1\,,
\eeq
hence
\beq
  \E\big[A_0^k\big]
= \pi^k r^{2 k}\,.
\eeq
Now we consider the case $1 < \ka < \infty$ ($r < R < \infty$). Here we have
\begin{align*}
  \EARk
= {} & \frac{2}{R^2}\int_0^R A^k(\rh)\,\rh\,\dd\rh
= \frac{2}{R^2}\left(\int_0^r A^k(\rh)\,\rh\,\dd\rh
+ \int_r^R A^k(\rh)\,\rh\,\dd\rh\right)\\
= {} & \frac{2}{R^2}\left(\int_0^r \pi^k\left(r^2-\rh^2\right)^k\rh\:
	\dd\rh + \int_r^R \pi^k\left(\rh^2-r^2\right)^k\rh\:\dd\rh\right)\\
= {} & \frac{2\pi^k r^{2k+2}}{R^2}
	\left(\int_0^1 \left(1-w^2\right)^k w\:\dd w 
+ \int_1^{R/r} \left(w^2-1\right)^k w\:\dd w\right),
\end{align*}
hence
\begin{align*}
  \EAk
= {} & \frac{2\pi^k r^{2k}}{\ka^2}\left(\frac{1}{2(k+1)} 
+ \int_1^\ka \left(w^2-1\right)^k w\:\dd w\right)
= \frac{2\pi^k r^{2k}}{\ka^2}\left(\frac{1}{2(k+1)} 
+ \frac{\left(\ka^2-1\right)^{k+1}}{2(k+1)}\right)\\
= {} & \frac{\pi^k r^{2k}}{k+1}\,\frac{1+(\ka^2-1)^{k+1}}{\ka^2}\,.
  \qedhere
\end{align*}

\end{proof}

The graph of $\E[A_\ka]/r^2$ is shown in Fig.\ \ref{Fig:ExpectedArea01}. For $\ka > 1$ we have
\beq
  \frac{\dd}{\dd\ka}\,\E[A_\ka]
= \frac{\pi r^2\left(\ka^4-2\right)}{\ka^3}\,.
\eeq
It follows that the expectation $\E[A_\ka]$ has its global minimum at point $\ka = \sqrt[4]{2} \approx 1.18921$ with value
\beq
  \E\left[A_{\sqrt[4]{2}}\right]
= \big(\sqrt{2}-1\big)\pi r^2
\approx 1.30129\,r^2\,.
\eeq
Furthermore one finds that 
\beq
  \E[A_1]
= \E\left[A_{\sqrt{2}}\right]
= \frac{1}{2}\,\pi r^2
= \frac{1}{2}\,\widetilde{A}(0)
= \frac{1}{2}\,\widetilde{A}\big(\sqrt{2}\,\big)
= \frac{1}{2}\,\E[A_0]\,.
\eeq
For the variance of $A_\ka$, $\Var[A_\ka]
= \E[A_\ka^2] - \E[A_\ka]^2$, we get
\beq
  \Var[A_\ka]
= \frac{\pi^2 r^4}{3}\,\frac{1-\left(1-\ka^2\right)^3}{\ka^2}
- \frac{\pi^2 r^4}{4}\,\frac{\left[1-\left(1-\ka^2\right)^2\right]^2}{\ka^4}
= \frac{\pi^2 r^4\ka^4}{12}
\eeq
if $0\le\ka\le 1$, and
\beq
  \Var[A_\ka]
= \frac{\pi^2 r^4}{3}\,\frac{1+\left(\ka^2-1\right)^3}{\ka^2}
- \frac{\pi^2 r^4}{4}\,\frac{\left[1+\left(\ka^2-1\right)^2\right]^2}{\ka^4}
= \pi^2 r^4 \left[-1 - \frac{1}{\ka^4} + \frac{2}{\ka^2} + \frac{\ka^4}{12}\right]
\eeq
if $1 < \ka < \infty$. One finds that $\Var[A_\ka]$ has local extrema at
\beq
  \ka_1 \approx 1.06840 \quad\mbox{and}\quad
  \ka_2 \approx 1.30621
\eeq
with values
\beq
  \Var[\ka_1] \approx 0.920036 \quad\mbox{and}\quad
  \Var[\ka_2] \approx 0.703487\,,
\eeq
respectively (see Fig. \ref{Fig:VarianceArea01}). 

\begin{figure}[H]
\begin{minipage}[c]{0.5\textwidth}
  \centering
	\includegraphics[scale=0.9]{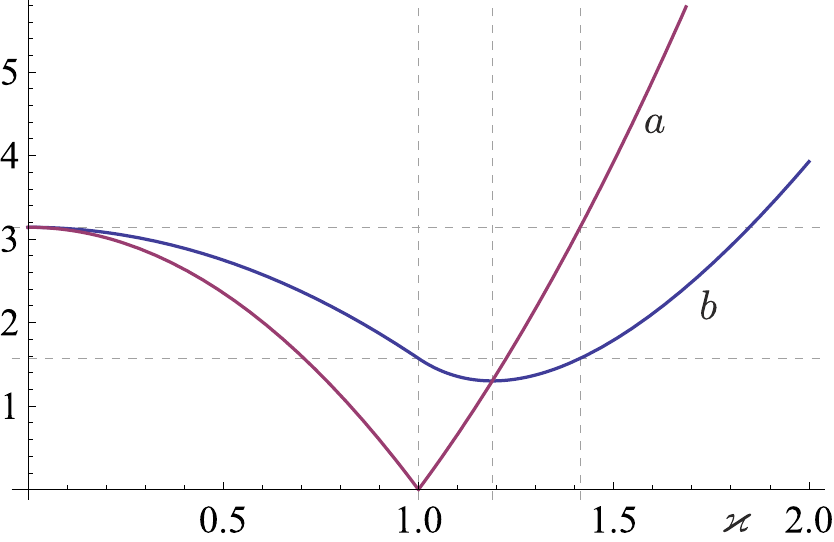}
	\vspace{-0.2cm}
	\caption{\label{Fig:ExpectedArea01} a) $\widetilde{A}(\ka)/r^2$ (see \eqref{Eq:A^*(w)}), b) $\E[A_\ka]/r^2$} 
\end{minipage}
\hfill
\begin{minipage}[c]{0.5\textwidth}
  \centering
	\includegraphics[scale=0.92]{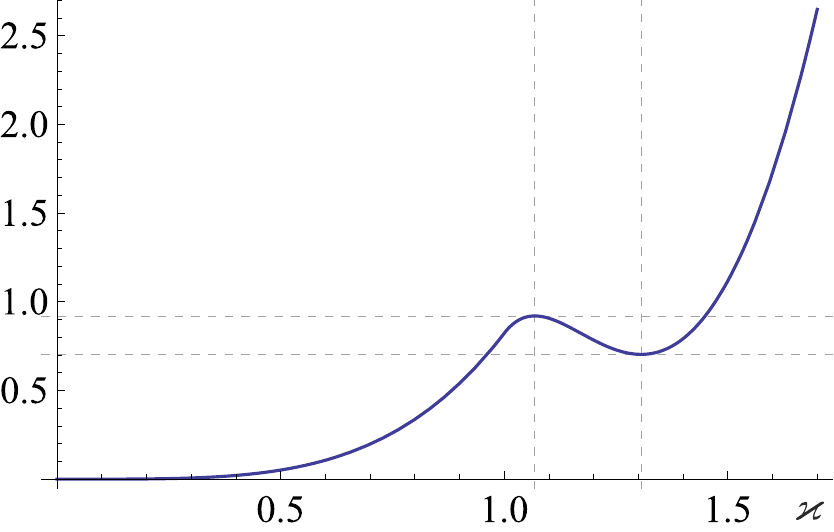}
	\vspace{-0.2cm}
	\caption{\label{Fig:VarianceArea01} $\Var[A_\ka]/r^4$} 
\end{minipage}
\end{figure}


\section{Distribution of the area}

Now we determine the distribution function
\beq
  F_\ka(x) = \W[A_\ka \le x]
\eeq
of the random variable $A_\ka$. We have to distinguish the following three cases.

\begin{itemize}[leftmargin=0.6cm] \setlength{\itemsep}{-1pt}

\item[1)] \underline{$0\le R\le r$ ($0\le\ka\le 1$):} The smallest area of an ellipse is equal to $\pi(r^2-R^2)$ and the biggest area equal to $\pi r^2$. We have $A_R > x$ if $P$ lies in an open disk with area $\pi\rh^2$.
From the first equation in \eqref{Eq:A(rh)}, with $x = A$ we get
\beq
  \rh^2 = r^2\,\bigg(1-\frac{x}{\pi r^2}\bigg)\,.
\eeq
It follows that
\beq
  \W[A_R > x]
= \frac{\pi\rh^2}{\pi R^2}
= \frac{r^2}{R^2}\,\bigg(1-\frac{x}{\pi r^2}\bigg)
\eeq
and
\beq
  \W[A_R \le x]
= 1 - \W[A_R > x]
= 1 - \frac{r^2}{R^2}\,\bigg(1 - \frac{x}{\pi r^2}\bigg)
\eeq
So we have
\beq
  F_\ka(x)
= \left\{\begin{array}{l@{\quad\mbox{if}\quad}l}
	0  &  -\infty < x < \pi r^2\left(1-\ka^2\right),\\[0.1cm]
	1 - \dfrac{1}{\ka^2}\,\bigg(1 - \dfrac{x}{\pi r^2}\bigg) & 
		\pi r^2\left(1-\ka^2\right) \le x < \pi r^2,\\[0.3cm]
	1 & \pi r^2 \le x < \infty\,.
  \end{array}\right.
\eeq

\item[2)] \underline{$r\le R\le\sqrt{2}\,r$ ($1\le\ka\le\sqrt{2}$):} For $0\le A_R\le\pi(R^2-r^2)$ we have $A_R > x$ if $P$ lies
\begin{itemize}[leftmargin=0.6cm] \setlength{\itemsep}{-1pt}
\item[a)] in an open disk with area $\pi\rh^2$ or
\item[b)] in an open annulus with area $\pi(R^2-\rh'^2)$, where $\rh'$ is the $\rh$ in the second equation in \eqref{Eq:A(rh)}.
\end{itemize}
It follows that
\begin{align*}
  \W[A_R > x]
= {} & \frac{\pi\rh^2 + \pi(R^2-\rh'^2)}{\pi R^2}
= 1 + \frac{\rh^2-\rh'^2}{R^2}
= 1 + \frac{r^2}{R^2}\,\bigg[1-\frac{x}{\pi r^2} 
- \bigg(1+\frac{x}{\pi r^2}\bigg)\bigg]\\
= {} &  1 - \frac{2x}{\pi R^2}\,, 
\end{align*}
hence
\beq
  \W[A_R \le x]
= 1 - \W[A_R > x]
= \frac{2x}{\pi R^2}\,.
\eeq 
For $\pi(R^2-r^2)\le A_R\le\pi r^2$ we have $A_R > x$ if $P$ lies in an open disk with area $\pi\rh^2$. Therefore, the distribution function is given by
\beq
  F_\ka(x)
= \left\{\begin{array}{l@{\quad\mbox{if}\quad}l}
	0  &  -\infty < x < 0\,,\\[0.2cm]
	\dfrac{2x}{\pi r^2 \ka^2} & 0 \le x < \pi r^2(\ka^2-1)\,,\\[0.3cm]
	1 - \dfrac{1}{\ka^2}\,\bigg(1 - \dfrac{x}{\pi r^2}\bigg) & 
		\pi r^2\left(\ka^2 - 1\right) \le x < \pi r^2,\\[0.4cm]
	1 & \pi r^2 \le x < \infty\,.
  \end{array}\right.
\eeq

\item[3)] \underline{$\sqrt{2}\,r\le R <\infty$ ($\sqrt{2}\le\ka <\infty$):} One easily finds
\beq
  F_\ka(x)
= \left\{\begin{array}{l@{\quad\mbox{if}\quad}l}
	0  &  -\infty < x < 0\,,\\[0.2cm]
	\dfrac{2x}{\pi r^2 \ka^2} & 0 \le x < \pi r^2\,,\\[0.3cm]
	\dfrac{1}{\ka^2}\,\bigg(1 + \dfrac{x}{\pi r^2}\bigg) & 
		\pi r^2 \le x < \pi r^2\left(\ka^2 - 1\right),\\[0.4cm]
	1 & \pi r^2\left(\ka^2 - 1\right) \le x < \infty\,.
  \end{array}\right.
\eeq

\end{itemize}

\noindent
Now we determine the moments of the random variabe $A_\ka$ in an alternative way:

\begin{itemize}[leftmargin=0.6cm] \setlength{\itemsep}{-1pt}

\item[1)] \underline{$0\le\ka\le 1$:} The density function is given by $f_\ka(x) = 1/(\pi r^2\ka^2)$ if $\pi r^2\left(1-\ka^2\right) \le x < \pi r^2$, and $f_\ka(x) = 0$ if the area is outside this interval. So we have
\begin{align*}
  \EAk
= {} & \int_{\pi r^2(1-\ka^2)}^{\pi r^2} x^k f_\ka(x)\:\dd x
= \frac{1}{\pi r^2\ka^2}\int_{\pi r^2(1-\ka^2)}^{\pi r^2} x^k\,
	\dd x
= \frac{\pi^k r^{2k}}{k+1}\,\frac{1 - (1-\ka^2)^{k+1}}{\ka^2}\,.
\end{align*}

\item[2)] \underline{$1\le\ka\le\sqrt{2}$\,:} The resctriction of the density function $f_\ka(x)$ to the interval $0\le x\le\pi r^2$ is given by
\beq
  f_\ka(x) = \frac{2}{\pi r^2 \ka^2} \;\;\mbox{if}\;\; 0\le x<\pi r^2(\ka^2-1)\,,
  \quad\mbox{and}\quad
  f_\ka(x) = \frac{1}{\pi r^2 \ka^2} \;\;\mbox{if}\;\; \pi r^2(\ka^2-1)\le x\le\pi r^2\,.
\eeq
It follows that
\begin{align*}
  \EAk
= {} & \frac{2}{\pi r^2 \ka^2}\int_0^{\pi r^2(\ka^2-1)} x^k\,\dd x
+ \frac{1}{\pi r^2 \ka^2}\int_{\pi r^2(\ka^2-1)}^{\pi r^2} x^k\,\dd x\\[0.1cm]
= {} & \frac{2\pi^k r^{2k}\left(\ka^2-1\right)^{k+1}}{(k+1)\ka^2}
+ \frac{\pi^k r^{2k}\left[1-\left(\ka^2-1\right)^{k+1}\right]}{(k+1)\ka^2} 
= \frac{\pi^k r^{2k}}{k+1}\,\frac{1+\left(\ka^2-1\right)^{k+1}}{\ka^2}\,.
\end{align*}

\item[3)] \underline{$\sqrt{2}\le\ka<\infty$:} One finds
\begin{align*}
  \EAk
= {} & \frac{2}{\pi r^2\ka^2}\int_0^{\pi r^2}x^k\,\dd x
+ \frac{1}{\pi r^2\ka^2}\int_{\pi r^2}^{\pi r^2(\ka^2-1)}x^k\,
	\dd x\\[0.1cm]
= {} & \frac{2\pi^k r^{2k}}{(k+1)\ka^2}
+ \frac{\pi^k r^{2k}\left[\left(\ka^2-1\right)^{k+1} - 1\right]}
	{(k+1)\ka^2}
= \frac{\pi^k r^{2k}}{k+1}\,\frac{1+\left(\ka^2-1\right)^{k+1}}{\ka^2}\,.
\end{align*}

\end{itemize}

\section{Moments of the perimeter}

By analogy to the determination of $\EARk$, we get the moments of the perimeter with 
\begin{align*}
  \EURk
= {} & \left(\int_{P \in D_R}u^k\,\dd P\right)\bigg/
  \left(\int_{P \in D_R}\dd P\right)
= \frac{1}{\pi R^2}\int_{\alpha=0}^{2\pi}\,\int_{\rh=0}^R
  u^k(\rh)\,\rh\,\dd\rh\,\dd\alpha\\
= {} & \frac{2}{R^2}\int_{\rh=0}^R u^k(\rh)\,\rh\,\dd\rh\,,
\end{align*}
where $u(\rh)$ is the perimeter of an ellipse generated by a point $P \in D_R$ with distance $\rh$ from $\Om$. The length of the semi-major axis is given by $r+\rh$, and the length of the semi-minor axis by $r-\rh$ or $\rh-r$. In both cases we have 
\beqn \label{Eq:u(rh)}
  u(\rh)
= 4(r+\rh)\:E\left(\frac{2\,\sqrt{\mathstrut r\rh}}{r+\rh}\right),
\eeqn
where
\beq
  E(k)
= \int_0^{\pi/2}\sqrt{1-k^2\sin^2\theta}\;\,\dd\theta
\eeq
is the complete elliptic integral of the second kind with modulus $k$, $0\le k\le 1$. So we have
\beqn \label{Eq:E[URk]}
  \EURk
= \frac{2^{2k+1}}{R^2}\bigintsss_{\,0}^R \rh\,(r+\rh)^k\,E^k\!\left(\frac{2\,\sqrt{\mathstrut r\rh}}{r+\rh}\right)\dd\rh\,.
\eeqn
By substituting $w = \rh/r$, the perimeter \eqref{Eq:u(rh)} may be written as
\beqn \label{Eq:u(w)}
  \tilde{u}(w) = rh(w)\,, \quad\mbox{where}\quad
  h(w) := 4(w+1)\:E\left(\frac{2\,\sqrt{\mathstrut w}}{w+1}\right),
\eeqn
and, with $U_\ka \equiv U_R$, $\ka = R/r$, the integral \eqref{Eq:E[URk]} becomes
\beqn \label{Eq:EUk_a}
  \EUk
= \frac{2^{2k+1}r^k}{\ka^2}\bigintsss_{\,0}^\ka w\,(w+1)^k\,E^k\!\left(\frac{2\,\sqrt{\mathstrut w}}{w+1}\right)\dd w\,.
\eeqn

\begin{thm} \label{Thm:Perimeter}
The expectation of the random perimeter $U_\ka = U_R$, $\ka = R/r$, of an ellipse generated by a random point $P \in D_R$ ($P$ uniformly distributed over the area of the disk $D_R$) is given by
\beq
  \E[U_\ka]
= \left\{
  \begin{array}{l@{\quad\mbox{if}\quad}l}
	2\pi r & \ka = 0\,,\\[0.1cm]
	\dfrac{8r}{9\ka^2}\left[\left(7\ka^2+1\right)E(\ka)
	+ \left(3\ka^4-2\ka^2-1\right)K(\ka)\right] & 0<\ka<1\,,\\[0.3cm]
	64r/9 & \ka=1\,,\\[0.15cm]
	\dfrac{8r}{9\ka}\left[\left(7\ka^2+1\right)E\!\left(\ka^{-1}\right)
	- 4\left(\ka^2-1\right)K\!\left(\ka^{-1}\right)\right] & \ka>1	
  \end{array}
  \right.
\eeq
where
\beq
  K(k)
= \int\limits_0^{\pi/2}\frac{\dd\theta}{\sqrt{1-k^2\sin^2\theta}}
\eeq
is the complete elliptic integral of the first kind with modulus $k$, $0\le k\le 1$. The expectations for the cases $0<\ka<1$ and $\ka>1$ can be subsumed in the formula
\beq
  \E[U_\ka]
= \frac{4(\ka+1)r}{9\ka^2}\left[\left(7\ka^2+1\right)E\left(\frac{2\,\sqrt{\mathstrut\ka}}{\ka+1}\right)
- (\ka-1)^2\,K\left(\frac{2\,\sqrt{\mathstrut\ka}}{\ka+1}\right)\right].
\eeq
\end{thm}

\begin{proof}
First, we consider the case $0 < \ka < 1$. Using the relation
\beqn \label{Eq:RelationE}
  E\left(\frac{2\,\sqrt{k}}{1+k}\right)
= \frac{1}{1+k}\left[2E(k)-\left(1-k^2\right)K(k)\right]
\eeqn
(\cite[Vol.\ 2, p.\ 299]{Gradstein_Ryshik}, Eq. (8.126.4)), \eqref{Eq:EUk_a} becomes
\beq
  \E[U_\ka]
= \frac{8r}{\ka^2}\left[2\int_0^\ka wE(w)\,\dd w
- \int_0^\ka wK(w)\,\dd w
+ \int_0^\ka w^3 K(w)\,\dd w\right].
\eeq
With
\begin{gather*}
  \int E(k)\,k\,\dd k
= \frac{1}{3}\left[\left(1+k^2\right)E(k)-(1-k^2)K(k)\right],\\
  \int K(k)\,k\,\dd k
= E(k)-(1-k^2)K(k)\,,\\
 \int K(k)\,k^3\,\dd k
= \frac{1}{9}\left[\left(4+k^2\right)E(k)-\left(1-k^2\right)
  \left(4+3k^2\right)K(k)\right]
\end{gather*}
(Equations (5.112.4), (5.112.3), (5.112.5) in \cite[Vol.\ 2, p.\ 13]{Gradstein_Ryshik}) we get
\beqn \label{Eq:EUR}
  \E[U_\ka]
= \frac{8r}{9\ka^2}\left[\left(7\ka^2+1\right)E(\ka)
	+ \left(3\ka^4-2\ka^2-1\right)K(\ka)\right],\quad
  0 < \ka < 1\,.
\eeqn
For $\ka=0$ we have $E(\ka)=\pi/2=K(\ka)$, hence
\beq
  \left(7\ka^2+1\right)E(\ka) + \left(3\ka^4-2\ka^2-1\right)K(\ka)
= 0\,.
\eeq
So $\E[U_\ka]$ has the indeterminate form $0/0$. Taking
\beq
  \frac{\dd E(k)}{\dd k}
= \frac{E(k)-K(k)}{k}
  \quad\mbox{and}\quad
  \frac{\dd K(k)}{\dd k}
= \frac{E(k)-(1-k^2)K(k)}{k(1-k^2)}
\eeq
(Equations (8.123.4), (8.123.2) in \cite[Vol.\ 2, p.\ 298]{Gradstein_Ryshik}) into account, applying L'H\^opital's rule twice gives
\beq
  \lim_{\ka\rightarrow 0}\E[U_\ka]
= 4r \lim_{\ka\rightarrow 0}\left[3E(\ka) 
+ 2\left(\ka^2-1\right)K(\ka)\right]
= 4r\left[3\,\frac{\pi}{2}-2\,\frac{\pi}{2}\right]
= 2\pi r\,.
\eeq
Now we consider the limit of \eqref{Eq:EUR} for $\ka\rightarrow 1$ (and hence $R\rightarrow r$). We have $E(1)=1$, $K(1)=\infty$. For $\ka=1$,
\beq
  \left(3\ka^4-2\ka^2-1\right)K(\ka)
\eeq
has the indeterminate form $0\cdot\infty$. {\em Mathematica} finds
\beq
  \lim_{\ka\rightarrow 1-}(3\ka^4-2\ka^2-1)K(\ka)
= 0\,.
\eeq
It follows that 
\beq
  \lim_{\ka\rightarrow 1-}\E[U_\ka]
= \frac{8r}{9\cdot 1}\left[(7+1)+0\right]
= \frac{64r}{9}\,.
\eeq
Now we consider the case $\ka = R/r > 1$. Here we have
\begin{align*}
  \E[U_\ka]
= {} & \frac{64r}{9\ka^2} + \frac{8}{R^2}\bigintsss_{\,\rh=r}^R
  \rh\,(r+\rh)\:
  E\left(\frac{2\,\sqrt{\mathstrut r\rh}}{r+\rh}\right)\dd\rh
  \displaybreak[0]\\[0.1cm]
= {} & \frac{64r}{9\ka^2} + \frac{8}{R^2}\bigintsss_{\,\rh=r}^R \rh^2
  \left(\frac{r}{\rh}+1\right)\,
  E\left(\frac{2\,\sqrt{\mathstrut r/\rh}}{(r/\rh)+1}\right)\dd\rh\,.
\end{align*}
The substitution
\beq
  v = \frac{r}{\rh}\,,\quad
  \dd v = -\frac{r}{\rh^2}\,\dd\rh\,,\quad
  \dd\rh = -\frac{r}{v^2}\,\dd v\,,\quad
  \rh = r \:\Longrightarrow\: v = 1\,,\quad
  \rh = R \:\Longrightarrow\: v = r/R = 1/\ka 
\eeq
gives
\beq
  \E[U_\ka]
= \frac{64r}{9\ka^2} + \frac{8r}{\ka^2}\bigintsss_{\,1/\ka}^1
  \frac{v+1}{v^4}\:E\left(\frac{2\,\sqrt{v}}{1+v}\right)\dd v\,.
\eeq
Applying \eqref{Eq:RelationE}, we get
\beq
  \E[U_\ka]
= \frac{64r}{9\ka^2} + \frac{8r}{\ka^2}\left[2\int_{1/\ka}^1
  \frac{E(v)}{v^4}\:\dd v - \int_{1/\ka}^1 \frac{K(v)}{v^4}\:\dd v
  + \int_{1/\ka}^1 \frac{K(v)}{v^2}\:\dd v 
  \right]  
\eeq
From \cite[Vol.\ 2, pp.\ 13-14]{Gradstein_Ryshik}, Equations (5.112.12), (5.112.9), we know that
\begin{gather*}
  \int\frac{E(k)}{k^4}\:\dd k = \frac{1}{9k^3}
  \left[2\left(k^2-2\right)E(k) + \left(1-k^2\right)K(k)\right],\\
  \int\frac{K(k)}{k^2}\:\dd k = -\frac{E(k)}{k}\,.
\end{gather*} 
{\em Mathematica} finds
\beq
  \int_{1/\ka}^1\frac{K(v)}{v^4}\:\dd v = \frac{1}{9}\left[
  -5 + \ka\left(\ka^2+4\right)E\!\left(\ka^{-1}\right)
  + 2\ka\left(\ka^2-1\right)K\!\left(\ka^{-1}\right)
  \right].
\eeq 
So we get
\beq
  \E[U_\ka]
= \frac{8r}{9\ka}\left[\left(7\ka^2+1\right)E\!\left(\ka^{-1}\right)
	- 4\left(\ka^2-1\right)K\!\left(\ka^{-1}\right)\right],\quad
  \ka > 1\,.
\eeq
For $\ka=1$,
\beq
  \left(\ka^2-1\right)K\!\left(\ka^{-1}\right)
\eeq
has the indeterminate form $0\cdot\infty$. {\em Mathematica} finds
\beq
  \lim_{\ka\rightarrow 1+}\left(\ka^2-1\right)K\!\left(\ka^{-1}\right)
= 0\,.
\eeq
It follows that
\beq
  \lim_{\ka\rightarrow 1+}\E[U_\ka]
= \frac{8r}{9\cdot 1}\left[(7+1)-0\right]
= \frac{64r}{9}\,.
\eeq
Using the relations \eqref{Eq:RelationE} and 
\beq
  K\left(\frac{2\,\sqrt{k}}{1+k}\right)
= (1+k)K(k)
\eeq
\cite[Vol.\ 2, p.\ 299]{Gradstein_Ryshik}, Eq. (8.126.3), for $0 < \ka < 1$ we get
\begin{align*}
f(\ka)
:= {} & \frac{4(\ka+1)r}{9\ka^2}\left[\left(7\ka^2+1\right)
  E\left(\frac{2\,\sqrt{\mathstrut\ka}}{\ka+1}\right) - (\ka-1)^2\,
  K\left(\frac{2\,\sqrt{\mathstrut\ka}}{\ka+1}\right)\right]\\
= {} & \dfrac{8r}{9\ka^2}\left[\left(7\ka^2+1\right)E(\ka)
	+ \left(3\ka^4-2\ka^2-1\right)K(\ka)\right] = \E[U_\ka]\,.
\end{align*}
For $\ka > 1$ we have
\begin{align*}
  E\left(\frac{2\,\sqrt{\mathstrut\ka}}{\ka+1}\right)
= {} & E\left(\frac{2\,\sqrt{\mathstrut\ka^{-1}}}{1+\ka^{-1}}\right)
= \frac{1}{1+\ka^{-1}}\left[2E\!\left(\ka^{-1}\right)
- \left(1-\ka^{-2}\right)K\!\left(\ka^{-1}\right)\right]\displaybreak[0]\\
= {} & \frac{2\ka}{\ka+1}\,E\!\left(\ka^{-1}\right)
- \frac{\ka-1}{\ka}\,K\!\left(\ka^{-1}\right)
\end{align*}
and
\beq
  K\left(\frac{2\,\sqrt{\mathstrut\ka}}{\ka+1}\right)
= K\left(\frac{2\,\sqrt{\mathstrut\ka^{-1}}}{1+\ka^{-1}}\right)
= \frac{\ka+1}{\ka}\,K\!\left(\ka^{-1}\right).
\eeq
It follows that
\begin{align*}
f(\ka)
= {} & \frac{4r}{9\ka^2}\left[2\ka\left(7\ka^2+1\right)
  E\!\left(\ka^{-1}\right)
- 8\ka\left(\ka^2-1\right)K\!\left(\ka^{-1}\right)\right]\\[0.2cm]
= {} & \frac{8r}{9\ka}\left[\left(7\ka^2+1\right)E\!\left(\ka^{-1}\right)
- 4\left(\ka^2-1\right)K\!\left(\ka^{-1}\right)\right] = \E[U_\ka]\,.
\end{align*}
For $\ka = 0$,
\beq
  \left(7\ka^2+1\right)
  E\left(\frac{2\,\sqrt{\mathstrut\ka}}{\ka+1}\right) - (\ka-1)^2\,
  K\left(\frac{2\,\sqrt{\mathstrut\ka}}{\ka+1}\right) = 0\,,
\eeq
hence
$f(0)$ has the indeterminate form $0/0$. For $\ka = 1$,
\beq
  (\ka-1)^2\,K\left(\frac{2\,\sqrt{\mathstrut\ka}}{\ka+1}\right)
\eeq
has the indeterminate form $0\cdot\infty$. {\em Mathematica} finds
\beq
  \lim_{\ka\rightarrow 0}f(\ka) = 2\pi r
  \qquad\mbox{and}\qquad
  \lim_{\ka\rightarrow 1}f(\ka) = \frac{64r}{9}\,. \qedhere
\eeq

\end{proof}

The graph of $\E[U_\ka]/r$ is shown in Fig.\ \ref{Fig:ExpectedPerimeter01}. {\em Mathematica} finds the series expansion
\beq
  \E[U_\ka] = \pi r\left(\frac{4\ka}{3} + \frac{1}{\ka}
- \frac{1}{16\ka^3}\right) + O\left(\ka^{-5}\right)
\eeq
about the point $\ka = \infty$. For abbreviation we put
\beqn \label{Eq:h(ka)}
  s(\ka)
:= \pi r\left(\frac{4\ka}{3} + \frac{1}{\ka} - \frac{1}{16\ka^3}\right).
\eeqn  
$s(\ka)$ provides a very good approximation for $\E[U_\ka]$ even for relatively small values of $\ka$ (see Fig.~\ref{Fig:ExpectedPerimeter01}). One finds that
\begin{gather*}
  \E[U_1] - s(1) \approx -2.29222\cdot 10^{-2}\,r\,,\qquad
  \E[U_2] - s(2) \approx -5.44238\cdot 10^{-4}\,r\,,\\[0.1cm]
  \E[U_{10}] - s(10) \approx -1.64009\cdot 10^{-7}\,r\,. 
\end{gather*}

\begin{figure}[H]
  \centering
	\includegraphics[scale=1.1]{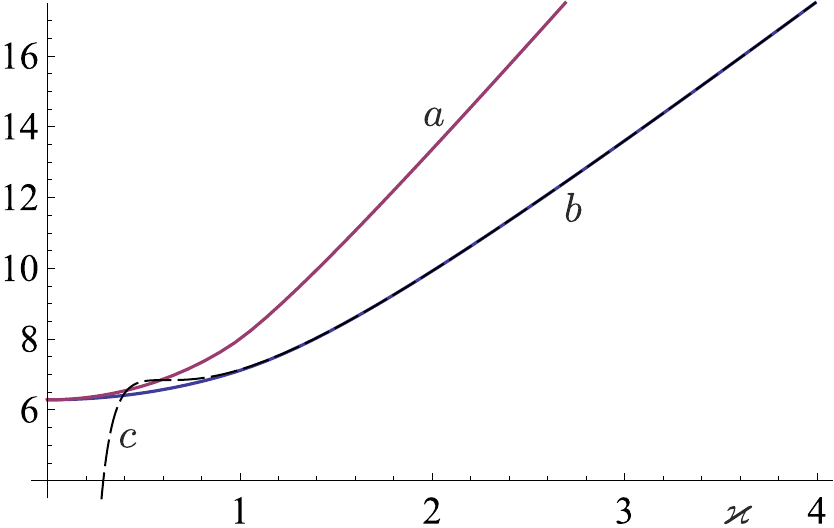}
	\vspace{-0.2cm}
	\caption{\label{Fig:ExpectedPerimeter01} a) $\tilde{u}(\ka)/r$ (see \eqref{Eq:u(w)}),\: b) $\E[U_\ka]/r$,\: c) $s(\ka)/r$} 
\end{figure}


\section{Distribution of the perimeter}

With \eqref{Eq:u(w)} the maximum perimeter in the disk of radius $R$ is given by
\beq
  \tilde{u}(\ka) = rh(\ka)
  \quad\mbox{with}\quad
  \ka = R/r\,.  
\eeq
Let $G_\ka(u) = \W[U_\ka \le u]$ be the distribution function of the random variable $U_\ka$. One easliy finds from geometrical considerations
\beqn \label{Eq:G}
  G_\ka(u)
= \left\{\begin{array}{l@{\quad\mbox{if}\quad}l}
	0 & -\infty < u < 2\pi r\,,\\[0.05cm]
	\dfrac{w^2(u)}{\ka^2} & 2\pi r \le u < rh(\ka)\,,\\[0.3cm]
	1 & rh(\ka) \le u < \infty\,, 
  \end{array}\right.
\eeqn
where $w = w(u)$ is the solution of
\beq
  rh(w) = u\,.
\eeq
The Figures \ref{Fig:Distribution_kappa=01}-\ref{Fig:Density_kappa=10} show examples for graphs of distribution functions $G_\ka$ and corresponding density functions $g_\ka$ (multiplied with $r$). The density functions are obtained by numerical differentiation of the distribution functions. For comparison the distribution function and the density function (multiplied with $r$) of the uniform distribution with support interval $2\pi r \le u \le rh(\ka)$ are shown (dashed lines).     

\begin{figure}[h]
\begin{minipage}[c]{0.5\textwidth}
  \centering
	\includegraphics[scale=0.9]{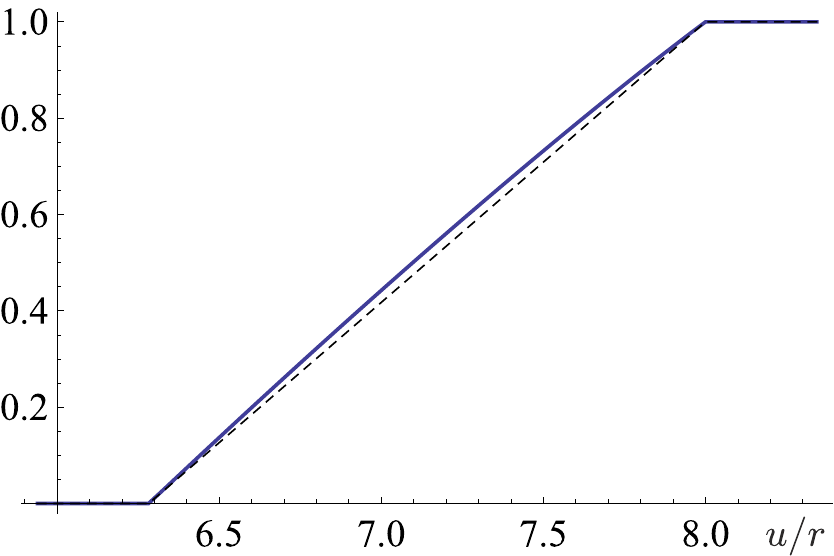}
	\vspace{-0.2cm}
	\caption{\label{Fig:Distribution_kappa=01} Distribution function $G_1(u)$} 
\end{minipage}
\hfill
\begin{minipage}[c]{0.5\textwidth}
  \vspace{0.125cm}
  \centering
	\includegraphics[scale=0.9]{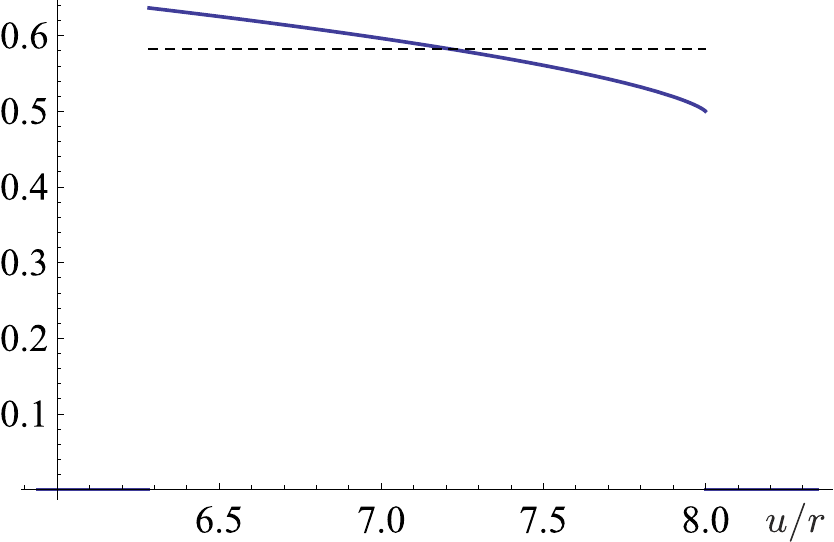}
	\vspace{-0.2cm}
	\caption{\label{Fig:Density_kappa=01} Density function $rg_1(u)$} 
\end{minipage}
\end{figure}

\begin{figure}[h]
\begin{minipage}[c]{0.5\textwidth}
  \centering
	\includegraphics[scale=0.9]{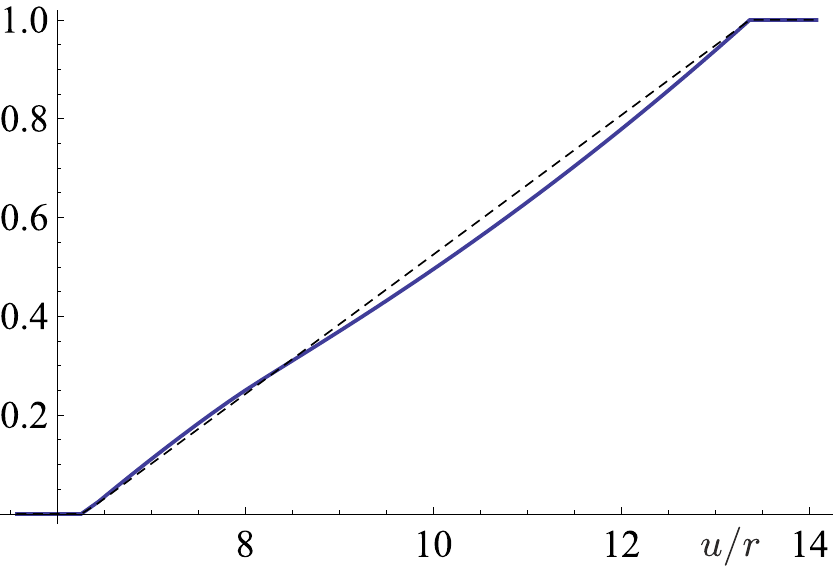}
	\vspace{-0.2cm}
	\caption{\label{Fig:Distribution_kappa=02} Distribution function $G_2(u)$} 
\end{minipage}
\hfill
\begin{minipage}[c]{0.5\textwidth}
  \vspace{0.2cm}
  \centering
	\includegraphics[scale=0.9]{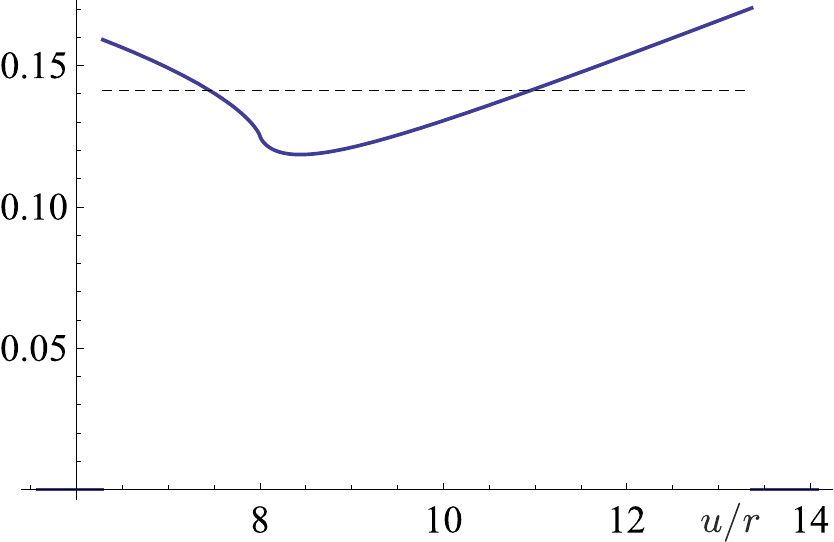}
	\vspace{-0.2cm}
	\caption{\label{Fig:Density_kappa=02} Density function $rg_2(u)$} 
\end{minipage}
\end{figure}

\begin{figure}[H]
\begin{minipage}[c]{0.5\textwidth}
  \centering
	\includegraphics[scale=0.9]{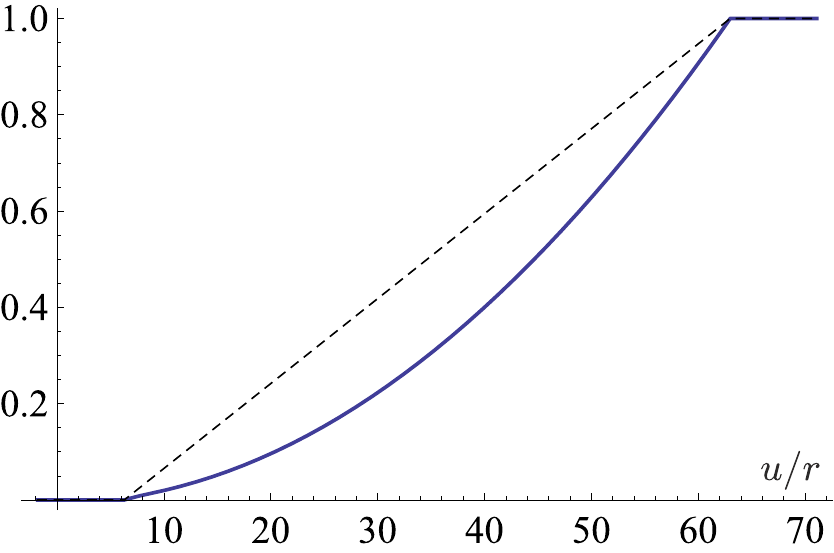}
	\vspace{-0.2cm}
	\caption{\label{Fig:Distribution_kappa=10} Distribution function $G_{10}(u)$} 
\end{minipage}
\hfill
\begin{minipage}[c]{0.5\textwidth}
  \vspace{0.32cm}
  \centering
	\includegraphics[scale=0.9]{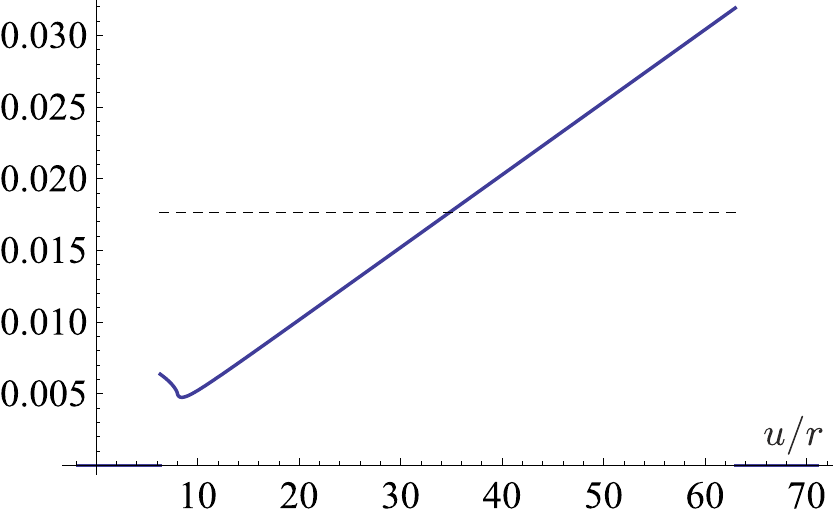}
	\vspace{-0.2cm}
	\caption{\label{Fig:Density_kappa=10} Density function $rg_{10}(u)$} 
\end{minipage}
\end{figure}

For the $k$-th moment of the perimeter $U_\ka$ we have the Stieltjes integral
\beq
  \EUk
= \int_{-\infty}^\infty u^k\,\dd G_\ka(u)
= \int_{2\pi r}^{rh(\ka)} u^k\,\dd G_\ka(u)\,.
\eeq
Integration by parts yields
\begin{align*}
  \EUk
= {} & u^k G_\ka(u)\,\Big|_{2\pi r}^{rh(\ka)} 
	- k\int_{2\pi r}^{rh(\ka)} u^{k-1}G_\ka(u)\,\dd u
= r^kh^k(\ka) - k\int_{2\pi r}^{rh(\ka)} u^{k-1}G_\ka(u)\,\dd u\,.
\end{align*}
From \eqref{Eq:G} it follows that
\beq
  \EUk
= r^kh^k(\ka) - \frac{k}{\ka^2}\int_{2\pi r}^{rh(\ka)} u^{k-1}w^2(u)\,
	\dd u\,.
\eeq
The substitution $u=rx$ gives
\beqn \label{Eq:EUk_b}
  \EUk
= r^k\left[h^k(\ka) - \frac{k}{\ka^2}\int_{2\pi}^{h(\ka)}x^{k-1}\widetilde{w}^2(x)\,\dd x\right],
\eeqn
where $\widetilde{w}=\widetilde{w}(x)$ is the solution of the equation
\beq
  rh(\widetilde{w}) = rx \quad\Longrightarrow\quad h(\widetilde{w}) = x
\eeq
for given value of $x\in[2\pi,h(\ka)]$.

Table \ref{Ta:Table} shows examples for numerical values of $\EUk/r^k$ which are obtained by numerical integration of \eqref{Eq:EUk_a} and \eqref{Eq:EUk_b} using {\em Mathematica}. The values for $k = 1$ also directly follow from Theorem \ref{Thm:Perimeter}.  

\begin{table}[H]
\begin{center}
\renewcommand{\arraystretch}{1.2}
\begin{tabular}{|r|l|l|} \hline
\rule{0pt}{12pt}
$k$ & \multicolumn{1}{c|}{$\E\big[U_2^k\big]/r^k$} & 
	\multicolumn{1}{c|}{$\E\big[U_3^k\big]/r^k$}\\[1pt] \hline\hline
 1 & 9.9232888058187711084              & 13.606226799878091189 \\
 2 & 102.96648551991466206              & 199.63369601685873413 \\
 3 & 1110.1715673108248830              & 3094.8106779481943393 \\
 4 & 12355.455260295394611              & 49903.060116964320575 \\
 5 & 141074.96324382298144              & 827860.92690590516817 \\
 6 & $1.6440752317143806830 \cdot 10^6$ & $1.4025517639333515857 \cdot 10^7$ \\
 7 & $1.9474969898992456378 \cdot 10^7$ & $2.4146079537555357859 \cdot 10^8$ \\
 8 & $2.3373182556510259688 \cdot 10^8$ & $4.2096527665994963840 \cdot 10^9$ \\
 9 & $2.8351332586002858086 \cdot 10^9$ & $7.4140462419039734138 \cdot 10^{10}$ \\
10 & $3.4691601000674476672 \cdot 10^{10}$ & $1.3167229871972133743 \cdot 10^{12}$ \\ \hline
\end{tabular}
\renewcommand{\arraystretch}{1}
\caption{Numerical values of $\EUk/r^k$ for $\ka = 2,\,3$}
\label{Ta:Table}
\end{center}
\end{table}



\bigskip\noindent
{\bf Uwe B\"asel}, Hochschule f\"ur Technik, Wirtschaft und Kultur Leipzig, Fakult\"at f\"ur Maschinenbau und Energietechnik, PF 30\,11\,66, 04251 Leipzig, Germany, \texttt{uwe.baesel@htwk-leipzig.de}

\end{document}